\documentclass[12pt]{article}
\usepackage{times}
\usepackage{booktabs}
\usepackage{appendix}
\usepackage{pifont}
\usepackage{floatrow}
\usepackage{caption}
\usepackage{mathrsfs}
\usepackage[fleqn]{amsmath}
\usepackage{amsfonts,amsthm,amssymb,mathrsfs,bbding}
\usepackage{txfonts}
\usepackage{graphics,multicol}
\usepackage{graphicx}
\usepackage{color}
\usepackage{caption}
\usepackage{subfigure}
\graphicspath{ {./images/} }
\usepackage{cite}
\usepackage{latexsym,bm}
\usepackage{longtable}
\usepackage{supertabular}
\usepackage{enumerate}
\usepackage{amsmath}
\evensidemargin=\oddsidemargin\topmargin=-1.5cm

\newtheorem{lem}{Lemma}[section]

\newtheorem{thm}{Theorem}[section]
\newtheorem{cor}{Corollary}[section]

\theoremstyle{definition}

\addtocounter{section}{0}
\allowdisplaybreaks[4]

\begin{document}
\title{Connected signed graphs with given inertia indices and given girth \footnote{This work is sponsored by Natural Science Foundation of Xinjiang Uygur Autonomous
Region (No. 2022D01A218).}}
\author{{BeiYan Liu,\ \   Fang Duan \footnote{Email: fangbing327@126.com}}\\[2mm]
\small School of Mathematics Science,  Xinjiang Normal University,\\
\small Urumqi, Xinjiang 830017, P.R. China}
\date{}
\maketitle {\flushleft\large\bf Abstract:}
Suppose that $\Gamma=(G, \sigma)$ is a connected signed graph with at least one cycle. The number of positive, negative and zero eigenvalues of the adjacency matrix of $\Gamma$ are called positive inertia index, negative inertia index and nullity of $\Gamma$, which are denoted by $i_+(\Gamma)$, $i_-(\Gamma)$ and $\eta(\Gamma)$, respectively. Denoted by $g$ the girth, which is the length of the shortest cycle of $\Gamma$. We study relationships between the girth and the negative inertia index of $\Gamma$ in this article. We prove $i_{-}(\Gamma)\geq \lceil\frac{g}{2}\rceil-1$ and extremal signed graphs corresponding to the lower bound are characterized. Furthermore, the signed graph $\Gamma$ with $i_{-}(\Gamma)=\lceil\frac{g}{2}\rceil$ for $g\geq 4$ are given. As a by-product, the connected signed graphs with given positive inertia index, nullity and given girth are also determined, respectively.

\begin{flushleft}
\end{flushleft}

\section{Introduction}
Let $G=(V(G), E(G))$ be a simple graph of order $|G|=n$ with vertex set $V(G)=\{v_1, v_2, ..., v_n\}$ and edge set $E(G)$. For $u, v\in V(G)$ and $S\subseteq V(G)$, $u\sim v$ denotes $uv\in E(G)$, $N_S(v)=\{u\in S\mid uv\in E(G)\}$ denotes the \emph{neighborhood} of $v$ in $S$, and $|N_S(v)|$ is called the \emph{degree} of $v$ in $S$, which is written as $d_S(v)$. In particular, if $S=V(G)$, then $N_S(v), d_S(v)$ can be replaced by $N(v), d(v)$, respectively. As usual, we denote by $K_{n_1, \ldots, n_l}$ the \emph{complete multipartite graph} with $l$ parts of sizes $n_1, \ldots, n_l$, and $C_n$, $P_n$ the \emph{cycle}, \emph{path} on $n$ vertices, respectively.

The \emph{adjacency matrix} $A(G)=(a_{ij})_{n\times n}$ of $G$ is defined as follows: $a_{ij}=1$ if $v_i$ is adjacent to $v_j$, and $a_{ij}=0$ otherwise. Since $A(G)$ is real and symmetric, all its eigenvalues are real, and so can be arranged as $\lambda_1(G)\geq  \lambda_2(G)\geq  \cdots \geq \lambda_n(G)$, which are also called the \emph{eigenvalues} of $G$. The number of positive, negative and zero eigenvalues of $A(G)$ are called \emph{positive inertia index}, \emph{negative inertia index} and \emph{nullity} of $G$, which are denoted by $i_+(G)$, $i_-(G)$ and $\eta(G)$, respectively. The applications of nullity and inertia indices have not only in discussion the stability of chemical molecules, but also in the mathematics itself. Hence they have been broadly investigated in the last decades. For some results on this topic, we refer the reader to \cite{Fang.D1,Fang.D2, M.R.Oboudi3,GuiHai.Yu}\, and references therein.

A \emph{signed graph} $\Gamma=(G, \sigma)$ consists of $G$ (called its \emph{underlying graph}) and a sign function $\sigma: E(G)\rightarrow \{+, -\}$. If $\sigma(e)=+$ (resp., $\sigma(e)=-$) for each $e\in E(G)$, then we denote the signed graph by $\Gamma=(G, +)$ (resp., $\Gamma=(G, -))$. Given a subset $\{u_1, u_2, \ldots ,u_t\}=S\subseteq V(\Gamma)$, the subgraph of $\Gamma$ induced by $S$, written as $\Gamma[S]$, is defined to be the signed graph with vertex set $S$ and edge set $\{u_iu_j\in E(\Gamma)\mid u_i\in S$ and $u_j\in S\}$, where the sign of $u_iu_j$ in $\Gamma[S]$ is the same as it in $\Gamma$. The \emph{adjacency matrix} of $\Gamma=(G, \sigma)$ is defined as $A(\Gamma)=(a_{u,v}^\sigma)_{u,v\in V(G)}$, where $a_{u,v}^\sigma=\sigma(uv) \cdot 1$ if $uv\in E(G)$ and $a_{u,v}^\sigma=0$ otherwise. The eigenvalues of $A(\Gamma)$ are called the \emph{eigenvalues} of $\Gamma$. The inertia indices, nullity and other basic notations of signed graphs are similarly to that of simple graphs.

The \emph{sign} of a cycle $C^\sigma$ is defined by $sgn(C^\sigma)=\prod_{e\in E(C)}\sigma(e)$. If $sgn(C^\sigma)=+$ (resp., $sgn(C^\sigma)=-$), then $C^\sigma$ is said to be \emph{positive} (resp., \emph{negative}). A signed graph $\Gamma$ is said to be \emph{balanced} if all of its cycles are positive, and \emph{unbalanced} otherwise. In particular, an acyclic signed graph is balanced. A function $\theta: V(\Gamma)\rightarrow \{+, -\}$ is called a \emph{switching function} of $\Gamma$. Switching $\Gamma$ by $\theta$, we obtain a new signed graph $\Gamma^\theta=(G, \sigma^\theta)$ whose sign function $\sigma^\theta$ is defined by $\sigma^\theta(xy)=\theta(x)\sigma(xy)\theta(y)$ for any $xy\in E(G)$. Two signed graphs $\Gamma_1$ and $\Gamma_2$ are said to be \emph{switching equivalent}, denoted by $\Gamma_1 \sim \Gamma_2$, if there exists a switching function $\theta$ such that $\Gamma_2=\Gamma_1^{\theta}$. By Sylvester's law of inertia, one can easily deduce that switching equivalent signed graphs share same inertia indices.

\begin{thm}[\cite{Yaoping.Hou}]\label{thm-1-1}
Let $\Gamma=(G, \sigma)$ be a signed graph. Then $\Gamma$ is balanced if and only if $\Gamma \backsim (G,+)$.
\end{thm}

\emph{Signed unicyclic graph} is a connected signed graph which has only one cycle. For an unbalanced signed unicyclic graph, the follow lemma holds.

\begin{lem}[\cite{FanY.Z}] \label{lem-2-11}
Let $\Gamma$ be an unbalanced signed unicyclic graph of order $n$. Then $\Gamma$ is switching equivalent to a signed unicyclic graph of order $n$ with exactly one (arbitrary) negative edge on the cycle.
\end{lem}

Signed graphs were introduced by Harary \cite{F.Harary}. Recently, there have been a number of investigations on the spectra of signed graphs, see \cite{Fang.D,W.H.Haemers,M.Brunetti,Q.Wu} and references therein. For more background on signed graphs we refer to \cite{F.Belardo} for which the authors make many conjectures about signed graphs. The simple connected graphs with given inertia indices and given girth are considered in \cite{Fang.D1,Fang.D2}. In this article, we will characterize connected signed graphs with inertia indices and given girth. The paper is organized as follows: In Section 2, some notions and lemmas are introduced. In Section 3,
we prove that $i_-(\Gamma)\geq \lceil\frac{g}{2}\rceil-1$ for a connected signed graph $\Gamma$ with given girth, and the extremal signed graphs corresponding to $\lceil\frac{g}{2}\rceil-1$ are completely characterized. Furthermore, we give the characterization for connected signed graph $\Gamma$ with $i_-(\Gamma)= \lceil\frac{g}{2}\rceil$ and $g\geq 4$. In Section 4 and Section 5, by elementary knowledge of linear algebra, the connected signed graphs with give positive inertia index, nullity and given girth are also determined, respectively.

\section{Preliminaries}
For two real symmetric matrices $A$ and $B$ of order $n$, if there exists a real invertible matrix $C$ such that $B=C^TAC$, then $B$ is called \emph{congruent} to $A$. The famous Sylvester's law of inertia states the relations of the inertia indices of two congruent matrices.

\begin{lem}\label{lem-2-0}
(Sylvester's law of inertia) If two real symmetric matrices $A$ and $B$ are congruent, then they have the same positive (resp., negative)
inertia index, the same nullity.
\end{lem}

\begin{lem}[\cite{R.A.Horn}]\label{lem-2-0-0}
Let $A$ be a real matrix of order $n$ and $\lambda_1, \lambda_2, \ldots, \lambda_n$ be all eigenvalues of $A$. Then $det(A)=\lambda_1 \lambda_2\cdots\lambda_n$.
\end{lem}

\begin{thm}[\cite{D.Cvetkovi}]\label{thm-2-0}
(Interlacing Theorem) Let $Q$ be a real $n\times m$ matrix such that $Q^TQ=I_m$ (where $m < n$), and let $A$ be an $n\times n$ real symmetric matrix
with eigenvalues $\lambda_1\geq \lambda_2\geq \cdots\geq \lambda_n$. If the eigenvalues of $B=Q^TAQ$ are $\mu_1\geq \mu_2\geq \cdots\geq \mu_m$, then
the eigenvalues of $B$ interlace those of $A$, that is, $\lambda_{n-m+i}\leq \mu_i\leq \lambda_i$ $(i=1,2, ..., m)$.
\end{thm}

The following Lemma \ref{lem-2-1-1} is obviously by Interlacing Theorem.

\begin{lem}\label{lem-2-1-1}
Let $\Gamma'$ be an induced signed subgraph of signed graph $\Gamma$. Then $i_+(\Gamma)\geq i_+(\Gamma')$ and $i_-(\Gamma)\geq i_-(\Gamma')$.
\end{lem}

\begin{lem}[\cite{GuiHai.Yu}, \cite{G.H.Yu}, \cite{G.H.Yu1}]\label{lem-2-1}
Let $C_n^\sigma$, $P_n^\sigma$ be a signed cycle, a signed path of order $n$, respectively.
\begin{enumerate}[(1)]
\item If $C_n^\sigma$ is balanced, then $i_-(C_n^\sigma)=\left\{\begin{array}{ll}\lceil\frac{n}{2}\rceil-1 &n\equiv 0,1(\bmod 4) \\ \lceil\frac{n}{2}\rceil & n\equiv 2,3(\bmod 4)\end{array}\right.$;
\item If $C_n^\sigma$ is unbalanced, then $i_-(C_n^\sigma)=\left\{\begin{array}{ll}\lceil\frac{n}{2}\rceil-1 &n\equiv 2,3(\bmod 4) \\ \lceil\frac{n}{2}\rceil & n\equiv 0,1(\bmod 4)\end{array}\right.$;
\item $i_-(P_n^\sigma)=\lfloor\frac{n}{2}\rfloor$.
\end{enumerate}
\end{lem}

\begin{lem}[\cite{G.H.Yu}]\label{lem-2-2}
Let $\Gamma$ be a signed graph containing a pendant vertex $u$, and let $\Gamma'$ be the induced signed subgraph of $\Gamma$ obtained by deleting $u$ together with the vertices adjacent to it. Then $i_+(\Gamma)=i_+(\Gamma')+1$ and $i_-(\Gamma)=i_-(\Gamma')+1$.
\end{lem}

\begin{lem}[\cite{Fang.D}]\label{lem-2-6}
Let $\Gamma=(G,\sigma)$ be a signed graph. Then $i_-(\Gamma)=1$ if and only if $\Gamma\cong $ $(K_{n_1,n_2}+sK_1, +)$, or $\Gamma\cong K^\sigma_{n_1,n_2,\ldots,n_l}+tK_1$ $(l\geq 3)$, where $K^\sigma_{n_1,n_2,\ldots,n_l}$ contains only unbalanced signed 3-cycle.
\end{lem}

Let $\Gamma=(G,\sigma)$ be a signed graph and $u, v\in V(\Gamma)$. Suppose that $N(u)=N(v)$ and there exists a switching function $\theta$ satisfying $\sigma^\theta(uw)= \sigma^\theta(vw)$ for all $w\in N(u)$. Then $u$ and $v$ are called \emph{twin vertices}. $\Gamma$ is called \emph{reduced} if there have no twin vertices in $\Gamma$, and \emph{non-reduced} otherwise. In \cite{Fang.D}, the authors derived the following Theorem \ref{thm-2-1}.

\begin{figure}[htbp]
    \centering
    \includegraphics[width=0.5\linewidth]{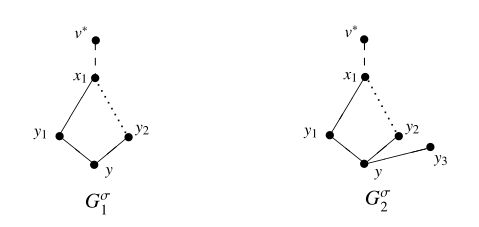}
    \caption{The signed graphs $G_1^\sigma$ and $G_2^\sigma$}
    \label{fig-1}
    \end{figure}

\begin{thm}[\cite{Fang.D}]\label{thm-2-1}
Let $\Gamma$ be a connected triangle-free reduced signed graph. Then $i_-(\Gamma)=2$ if and only if $\Gamma$ is isomorphic to one of $P_4^\sigma$, $P_5^\sigma$, unbalanced $C_4^\sigma$, balanced $C_5^\sigma$, unbalanced $C_6^\sigma$, $G_1^\sigma$, $G_2^\sigma$ (see Fig. \ref{fig-1}).
\end{thm}

From Theorem \ref{thm-2-1}, it follows Corollary \ref{cor-1}.

\begin{cor}\label{cor-1}
Let $\Gamma$ be a connected signed graph with girth $g=4$.
\begin{enumerate}[(1)]
\setlength{\itemsep}{0pt}
\item If $\Gamma$ contains an unbalanced 4-cycle, then $i_-(\Gamma)=2$ if and only if $\Gamma$ can be obtained from $G_1^\sigma$, $G_2^\sigma$ or unbalanced $C_4^\sigma$ by adding twin vertices;
\item If $\Gamma$ contains no unbalanced 4-cycle, then $i_-(\Gamma)=2$ if and only if $\Gamma$ can be obtained from $P_4^\sigma$, $P_5^\sigma$, balanced $C_5^\sigma$ or unbalanced $C_6^\sigma$ by adding twin vertices.
\end{enumerate}
\end{cor}
\begin{proof}
 The sufficiency of (1) and (2) are clearly true by Theorem \ref{thm-2-1}. Conversely, note that adding twin vertices to $P_4^\sigma$, $P_5^\sigma$, balanced $C_5^\sigma$ and unbalanced $C_6^\sigma$ do not produce unbalanced 4-cycle. Then, (1) holds since unbalanced $C_4^\sigma$, $G_1^\sigma$ and $G_2^\sigma$ already have unbalanced 4-cycle. If $\Gamma$ contains no unbalanced 4-cycle, then $\Gamma$ must contain a balanced 4-cycle. In a similar way of (1) it can be shown (2).
\end{proof}

Let $C^\sigma$ be a signed cycle of $\Gamma$ and $v\in V(\Gamma)\setminus V(C^\sigma)$. Denoted by $d(v, x)$ the length of a shortest path between $v$ and $x$ and let $d(v, C^\sigma)=min\{d(v,x)\mid x\in V(C^\sigma)\}$. Defined
$$N_r(v,C^\sigma) =\left\{v\in V\left(\Gamma\right)\setminus V( C^\sigma)\mid d(v,C^\sigma)=r\right\},$$
where $r$ is a positive integer, and $|N_r(v, C^\sigma)|$ denote the number of vertices in $N_r(v, C^\sigma)$. The following two lemmas are proved in \cite{Fang.D1} for simple graphs. For completeness of this paper, we present the proof here for signed graphs by a similar manner.

\begin{lem}\label{lem-2-5}
Let $\Gamma=(G, \sigma)$ be a signed graph with girth $g$ and $C_g^\sigma$ be a shortest cycle of $\Gamma$. Suppose $y,y'\in V(C_g^\sigma)$ and there exists a path $P^\sigma$ of length $t$ from $y$ to $y'$ vertex-disjoint to $C_g^\sigma$. Then $\lceil\frac{g}{2}\rceil\leq t$. Furthermore, if $g\geq 5$, then each vertex in $V(\Gamma)\setminus V(C_g^\sigma)$ is adjacent to at most one vertex of $V(C_g^\sigma)$.
\end{lem}
\begin{proof}
There are two paths from $y$ to $y'$ in $C_g^\sigma$ and the length of the shorter one, say $P_1^\sigma$, at most $\lfloor\frac{g}{2}\rfloor$. We can found a cycle of length at most $\lfloor\frac{g}{2}\rfloor+t$ in $\Gamma$, which consist of $P_1^\sigma$ and $P^\sigma$. Then, $\lfloor\frac{g}{2}\rfloor+t\geq g$, and hence, the desired inequality holds. In addition, if $g\geq 5$ and $x\in V(\Gamma)\setminus V(\ C_g^\sigma)$ is adjacent to two vertices $y_1, y_2\in V(C_g^\sigma)$, then there exists a path of length $2$ from $y_1$ to $y_2$ vertex-disjoint to $C_g^\sigma$. By using the first assertion, we have $\lceil\frac{g}{2}\rceil\leq 2$. This contradicts the choice of $g$, and we are done.
\end{proof}

\begin{lem}\label{lem-2-7}
Let $\Gamma=(G,\sigma)$ be a signed graph with girth $g$ and $C_g^\sigma$ be a shortest cycle of $\Gamma$. If $i_-(\Gamma)=i_-(C_g^\sigma)$, then, for $r\geq 2$, we have $N_r(v,C_g^\sigma)=\emptyset$.
\end{lem}
\begin{proof}
On the contrary, assume that $x'\in N_2(v,C_g^\sigma)\neq \emptyset$ and $x'\sim x \in N_1(v,C_g^\sigma)$. Then, $i_-(\Gamma)\geq i_-(\Gamma[V(C_g^\sigma)\cup \{x',x\}])=i_-(C_g^\sigma)+1>i_-(C_g^\sigma)$ by Lemmas \ref {lem-2-1-1} and \ref {lem-2-2}, a contradiction. We conclude $N_2(v,C_g^\sigma)=\emptyset$, hence, $N_r(v,C_g^\sigma)=\emptyset$ for $r\geq 3$.
\end{proof}

\section{Connected signed graphs with given negative inertia index and given girth}
The following theorem establishes a fundamental inequality  relating the negative inertia index of signed graphs and it's girth.
\begin{thm}\label{thm-3-1}
Let $\Gamma=(G,\sigma)$ be a connected signed graph with girth $g$ and $C_g^\sigma$ be a shortest cycle in $\Gamma$. Then the inequality $i_-(\Gamma)\geq\lceil\frac{g}{2}\rceil-1$ holds. Equality holds if and only if $\Gamma$ satisfies one of the following conditions:
\begin{enumerate}[(1)]
\setlength{\itemsep}{0pt}
\item $\Gamma\cong C_g^\sigma$, where $C_g^\sigma$ is balanced and $g\equiv 0,1(\bmod\ 4)$, or unbalanced and $g\equiv 2,3(\bmod\ 4)$;
\item $\Gamma\cong (K_{n_1,n_2}, +)$, or $K^\sigma_{n_1,n_2,\ldots,n_l}$ $(l\geq 3)$, where $K^\sigma_{n_1,n_2,\ldots,n_l}$ contains only unbalanced signed 3-cycle.
\end{enumerate}
\end{thm}

\begin{proof}
It is easy to see that $C_g^\sigma$ is an induced signed subgraph of $\Gamma$. Since $i_-(C_g^\sigma)\geq\lceil\frac{g}{2}\rceil-1$ by Lemma \ref{lem-2-1} and $i_-(\Gamma)\geq i_-(C_g^\sigma)$ by Lemma \ref{lem-2-1-1}, the equality $i_-(\Gamma)\geq\lceil\frac{g}{2}\rceil-1$ holds. Next, we treat the equality case.

The signed graphs in (1) and (2) satisfy the equality by Lemmas \ref{lem-2-1} and \ref{lem-2-6}. Conversely, let $\Gamma$ be a signed graph satisfying $i_-(\Gamma)=\lceil\frac{g}{2}\rceil-1$. Suppose that $\Gamma$ is a cycle. Then, $\Gamma\cong C_g^\sigma$, and $g\equiv 0,1(\bmod\ 4)$ if $\Gamma$ is balanced, $g\equiv 2,3(\bmod\ 4)$ if $\Gamma$ is unbalanced following from Lemma \ref{lem-2-1}. Hence we may assume that $\Gamma$ is not a cycle and set $x\in N_1(v,C_g^\sigma)\neq\emptyset$. Moreover, because $i_-(\Gamma)=i_-(C_g^\sigma)=\lceil\frac{g}{2}\rceil-1$, we have
$N_r(v,C_g^\sigma)=\emptyset$ for $r\geq 2$ following from Lemma \ref{lem-2-7}. Suppose that $N_{C_g^\sigma}(x)=\{y\}$. According to Lemmas \ref{lem-2-1-1}, \ref{lem-2-1} and \ref{lem-2-2}, a contradiction
$$i_-(\Gamma)\geq i_-(\Gamma[V(C_g^\sigma)\cup\{x\}])=i_-(P_{g-1}^\sigma)+1=\left \lfloor \frac{g-1}{2}\right \rfloor+1=\left \lceil \frac{g}{2} \right \rceil$$
arise. Suppose that $|N_{C_g^\sigma}(x)|\geq 2$ and let $y,y'\in N_{C_g^\sigma}(x)$. Note that there exists a path of length $2$ from $y$ to $y'$ vertex-disjoint to $C_g^\sigma$  in this case. We have $\lceil\frac{g}{2}\rceil\leq 2$ by virtue of Lemma \ref{lem-2-5}, which implies $g=3$ or 4. Hence $i_-(\Gamma)=\left\lceil\frac{g}{2}\right \rceil-1=1$. Therefore, by Lemma \ref{lem-2-6}, $\Gamma$ is isomorphic to $(K_{n_1,n_2}, +)$, or $K^\sigma_{n_1,n_2,\ldots,n_l}$ $(l\geq 3)$, where $K^\sigma_{n_1,n_2,\ldots,n_l}$ contains only unbalanced signed 3-cycle.
\end{proof}

Let $C^\sigma$ be the unique cycle of a signed unicyclic graph $\Gamma$. If $\Gamma$ satisfies (a) $N_r(v, C^\sigma)=\emptyset$ for $r\geq 2$;
(b) $N_1(v, C^\sigma)$ is an independent set, and each vertex of $N_1(v, C^\sigma)$ is adjacent to exactly one vertex of $V(C^\sigma)$, then $\Gamma$ is called a \emph{canonical signed unicyclic graph}. Set $H^\sigma$ is an induced signed subgraph of $\Gamma$. We call $H^\sigma$ a \emph{pendant star} if $H^\sigma$ is a star and its center is the only vertex which have exactly two neighbors in $C^\sigma$, this center vertex of $H^\sigma$ is called the \emph{major vertex}.

We given a way of depicting signed graphs that will be used in this section. Two vertices joining with a plain line (resp., dotted line) denotes that this edge is positive (resp., negative), and with a dashed line denotes that the positivity and negativity of this edge is uncertain. Two canonical signed unicyclic graphs $K_1^\sigma$ and $K_2^\sigma$ are given in Fig. \ref{fig-2}, where $K_1^\sigma$ has girth 7 and $i_-(K_1^\sigma)=\lceil \frac{7}{2}\rceil=4$, $K_2^\sigma$ has girth 8 and $i_-(K_2^\sigma)=\lceil \frac{8}{2}\rceil=4$. In addition, $K_1^\sigma[v, v_1, \ldots ,v_t]$ and $K_2^\sigma[v, v_1, \ldots ,v_t]$ are pendant stars of $K_1^\sigma$ and $K_2^\sigma$, respectively. By a method similar to that of \cite{Fang.D1}, we characterizes the canonical signed unicyclic graph $\Gamma$ with $i_-(\Gamma)=\lceil \frac{g}{2}\rceil$ in Theorem \ref{thm-3-2}.

	\begin{figure}[t]
		\centering % ±íʾ¾ÓÖÐ
		\subfigure[The canonical signed unicyclic graph with girth $7$ and $i_-(K_1^\sigma)=4$]{%
			\resizebox*{5cm}{!}{\includegraphics{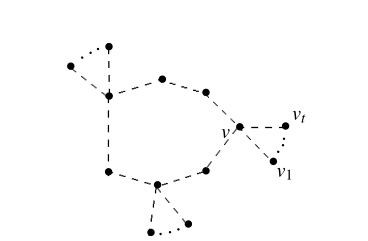}}}\hspace{5pt}
		\subfigure[The canonical signed unicyclic graph with girth $8$ and $i_-(K_2^\sigma)=4$]{%
			\resizebox*{5cm}{!}{\includegraphics{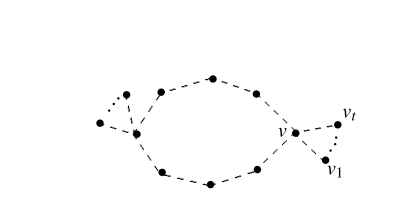}}}
		\caption{The  canonical signed unicyclic graphs $K_1^\sigma$ and $K_2^\sigma$} \label{fig-2}
	\end{figure}

\begin{thm}\label{thm-3-2}
Let $\Gamma$ be a canonical signed unicyclic graph with girth $g$ and the unique cycle $C_g^\sigma$. Then, the following statements hold:
\begin{enumerate}[(1)]
\setlength{\itemsep}{0pt}
\item If $\Gamma$ is a cycle, then $i_-(\Gamma)=\lceil \frac{g}{2}\rceil$ if and only if $\Gamma\cong C_g^\sigma$, where $C_g^\sigma$ is balanced and $g\equiv 2,3(\bmod\ 4)$, or unbalanced and $g\equiv 0,1(\bmod\ 4)$;
\item If $\Gamma$ is not a cycle and $g\equiv 1,3(\bmod\ 4)$, then $i_-(\Gamma)=\lceil \frac{g}{2}\rceil$ if and only if $\Gamma$ has one or more attached pendant stars such that exactly one path between any two major vertices of $V(C_g^\sigma)$ has even order;
\item If $\Gamma$ is not a cycle and $g\equiv 0,2(\bmod\ 4)$, then $i_-(\Gamma)=\lceil \frac{g}{2}\rceil$ if and only if $\Gamma$ has one or more attached pendant stars such that all paths between any two major vertices of $V(C_g^\sigma)$ have odd order.
\end{enumerate}
\end{thm}
\begin{proof}
If $\Gamma$ is a cycle, then by Lemma \ref{lem-2-1}, $i_-(\Gamma)=\lceil \frac{g}{2}\rceil$ if and only if $\Gamma\cong C_g^\sigma$, where $C_g^\sigma$ is balanced and $g\equiv 2,3(\bmod\ 4)$, or unbalanced and $g\equiv 0,1(\bmod\ 4)$. Next, we consider the case which $\Gamma$ is not cycle.

Suppose that there are $k\geq 1$ pendant stars in $\Gamma$ and let $P_{l_1}^\sigma, \ldots, P_{l_k}^\sigma$ be the paths obtained upon removing the $k$ pendant stars from $\Gamma$. Then, $g=k+l_1 +\cdots +l_k$. On the other hand, because there are $k$ pendant stars in $\Gamma$, we have $k+i_-(P_{l_1}^\sigma)+\cdots +i_-(P_{l_k}^\sigma)=i_-(\Gamma)=\lceil \frac{g}{2}\rceil$ by Lemma \ref{lem-2-2}. Let us consider the following two cases:
{\flushleft\bf Case 1.} $g\equiv 1,3$ (mod 4).

Since $g$ is an odd, we have $i_-(\Gamma)=\lceil \frac{g}{2}\rceil$ if and only if
$$k+i_-(P_{l_1}^\sigma)+\cdots +i_-(P_{l_k}^\sigma)=i_-(\Gamma)=\lceil \frac{g}{2}\rceil=\frac{g+1}{2}=\frac{k+l_1 +\cdots +l_k+1}{2}.$$
Then,
$$k-1=[l_1-2i_-(P_{l_1}^\sigma)]+[l_2-2i_-(P_{l_2}^\sigma)]+\cdots +[l_k-2i_-(P_{l_k}^\sigma)].$$
According to Lemma \ref{lem-2-1}, $l_j-2i_-(P_{l_j}^\sigma)=0$ if $l_j$ is even and $l_j-2i_-(P_{l_j}^\sigma)=1$ if $l_j$ is odd for $j\in \{1,\ldots,k\}$. Consequently,  $i_-(\Gamma)=\lceil \frac{g}{2}\rceil$ if and only if one of $l_1, \ldots, l_k$ is even.

{\flushleft\bf Case 2.} $g\equiv 0, 2$ (mod 4).

Since $g$ is an even, we have $i_-(\Gamma)=\lceil \frac{g}{2}\rceil$ if and only if
$$k+i_-(P_{l_1}^\sigma)+\cdots +i_-(P_{l_k}^\sigma)=i_-(\Gamma)=\lceil \frac{g}{2}\rceil=\frac{g}{2}=\frac{k+l_1 +\cdots +l_k}{2}.$$
Then,
$$k=[l_1-2i_-(P_{l_1}^\sigma)]+[l_2-2i_-(P_{l_2}^\sigma)]+\cdots +[l_k-2i_-(P_{l_k}^\sigma)].$$
By a similar discussion as Case 1, $i_-(\Gamma)=\lceil \frac{g}{2}\rceil$ holds if and only if all of $l_1,\ldots, l_k$ are odd. It follows the desired conclusion.
\end{proof}

For those who are not canonical signed unicyclic graphs, we can only characterize signed graph $\Gamma$ with $i_-(\Gamma)=\lceil\frac{g}{2}\rceil$ and $g\geq 4$.

\begin{figure}[t]
    \centering
    \includegraphics[width=0.75\linewidth]{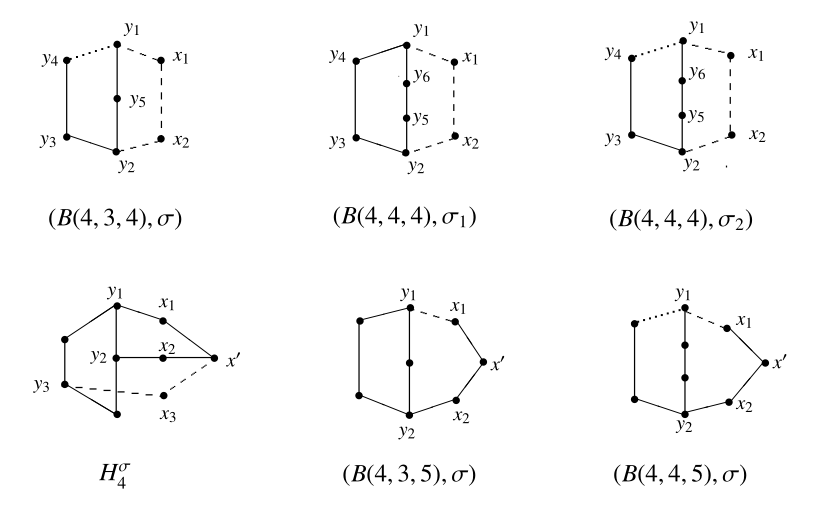}
    \caption{The signed graphs $(B(4,3,4),\sigma)$, $(B(4,4,4),\sigma_1)$, $(B(4,4,4),\sigma_2)$, $H_4^\sigma$, $(B(4,3,5),$ $\sigma)$ and $(B(4,4,5),\sigma)$}
    \label{fig-3}
    \end{figure}

\begin{thm}\label{thm-3-3}
Let $\Gamma=(G,\sigma)$ be a connected signed graph with girth $g\geq 4$, and there exists a cycle $C_g^\sigma$ satisfying $i_-(C_g^\sigma)=\lceil\frac{g}{2}\rceil$. Suppose that $\Gamma$ is not canonical signed unicyclic graph. Then $i_-(\Gamma)=\lceil\frac{g}{2}\rceil$ if and only if $\Gamma$ is isomorphic to one of the following signed graphs:
\begin{enumerate}[(1)]
\setlength{\itemsep}{0pt}
\item The signed graphs obtained from unbalanced $C_4^\sigma$, $G_1^\sigma$ or $G_2^\sigma$ by adding twin vertices;
\item The signed graphs $\Gamma_1, \Gamma_2$ and $\Gamma_3$ (see Fig. \ref{fig-5}).
\end{enumerate}
\end{thm}
\begin{proof}
The connected signed graph $\Gamma$ with girth $g\geq 4$ satisfies $i_-(\Gamma)=\lceil\frac{g}{2}\rceil=2$ if and only if $g=4$. Because there exists an unbalanced 4-cycle in $\Gamma$ due to hypothesis. By Corollary \ref{cor-1} (1), $\Gamma$ can be obtained from unbalanced $C_4^\sigma$, $G_1^\sigma$ or $G_2^\sigma$ by adding twin vertices. It suffice to show the case $g\geq 5$ as follows.

The signed graphs shown in (2) satisfy the equality by a simple operation. Now let $\Gamma=(G,\sigma)$ be a connected signed graph with girth $g\geq 5$. Since there exists a cycle $C_g^\sigma$ in $\Gamma$ satisfying $i_-(C_g^\sigma)=\lceil\frac{g}{2}\rceil=i_-(\Gamma)$, we have $N_r(v,C_g^\sigma)$ $=\emptyset$ for $r\geq 2$ due to Lemma \ref{lem-2-7}. Furthermore, by Lemma \ref{lem-2-5}, the vertex of $N_1(v,C^\sigma_g)=V(\Gamma)\setminus V(C^\sigma_g)$ is adjacent to exactly one vertex of $V(C^\sigma_g)$. Hence $N_1(v,C_g^\sigma)$ is not an independent set since otherwise $\Gamma$ is a canonical signed unicyclic graph. We may assume that $x_1$ is adjacent to $x_2$ in $N_1(v,C_g^\sigma)$ and set $x_1\sim y_1\in V(C_g^\sigma)$ and $x_2\sim y_2\in V(C_g^\sigma)$. Now there exists a path of length 3 from $y_1$ to $y_2$ vertex-disjoint to $C_g^\sigma$. Then $\lceil\frac{g}{2}\rceil\leq 3$ by Lemma \ref{lem-2-5}, and so $g\leq 6$. Note that the condition $i_-(C_g^\sigma)=\lceil\frac{g}{2}\rceil$ implies that $g\equiv 2,3(\bmod\ 4)$ if $C_g^\sigma$ is balanced and $g\equiv 0,1(\bmod\ 4)$ if $C_g^\sigma$ is unbalanced by Lemma \ref{lem-2-1}. The following two cases can be distinguish:

{\flushleft\bf Case 1.} $g=5$ and $C_5^\sigma$ is unbalanced.

We may assume that exactly one edge of $C_5^\sigma$, say $y_1y_4$, is negative by Lemma \ref{lem-2-11}. If $N_1(v,C_5^\sigma)=\{x_1, x_2\}$, then $\Gamma\cong (B(4,3,4), \sigma)$ (see Fig. \ref{fig-3}). Note that the induce signed subgraph $\Gamma[V(C_5^\sigma)\cup \{x_1\}]$ of $\Gamma$ has three positive and negative eigenvalues, respectively. We obtain $i_+(\Gamma)\geq 3$ and $i_-(\Gamma)\geq 3$ by Lemma \ref{lem-2-1-1}.  In order of $y_1, y_2, y_3, y_4, y_5, x_1,x_2$, $A(\Gamma)$ can be written as

$$A(\Gamma)=\left(
        \begin{array}{ccccccc}
          0 & 0 & 0 & -1 & 1 & a_{16}^\sigma & 0\\
          0 & 0 & 1 & 0 & 1 & 0 & a_{27}^\sigma\\
          0 & 1 & 0 & 1 & 0 & 0 & 0\\
          -1 & 0 & 1 & 0 & 0 & 0 & 0\\
          1 & 1 & 0 & 0 & 0 & 0 & 0\\
          a_{16}^\sigma & 0 & 0 & 0 & 0 & 0 & a_{67}^\sigma\\
          0 & a_{27}^\sigma & 0 & 0 & 0 & a_{67}^\sigma & 0\\
        \end{array}
      \right).$$
By a simple operation, we get $det(A(\Gamma))=2(a_{67}^\sigma)^2-2a_{16}^\sigma a_{27}^\sigma a_{67}^\sigma$. Then $i_-(\Gamma)=4>\lceil \frac{5}{2}\rceil$ if $a_{16}^\sigma a_{27}^\sigma a_{67}=-1$ and $i_-(\Gamma)=3=\lceil \frac{5}{2}\rceil$ if $a_{16}^\sigma a_{27}^\sigma a_{67}^\sigma=1$ by Lemma \ref{lem-2-0-0}. Therefore, $\Gamma$ is switching equivalent to $\Gamma_1$ (see Fig. \ref{fig-5}).

If $N_1(v,C_5^\sigma)\setminus\{x_1, x_2\}\not=\emptyset$, then $\Gamma[V(C_5^\sigma)\cup\{x_1,x_2\}]\cong \Gamma_1$ by the above discussion. Assume that any vertex $x$ of $N_1(v,C_5^\sigma)\setminus \{x_1, x_2\}$ is adjacent to neither $x_1$ nor $x_2$. Then $x$ can only be adjacent to $y_5\in V(C_5^\sigma)$ since otherwise $i_-(\Gamma)\geq i_-(\Gamma[V(\Gamma_1)\cup \{x\}])=4>\lceil\frac{5}{2}\rceil$, which implies that $N_1(v,C_5^\sigma)\setminus \{x_1, x_2\}$ is an independent set and $\Gamma\cong \Gamma_2$ (see Fig. \ref{fig-5}). Now let $x_3\in N_1(v,C_5^\sigma)\setminus \{x_1, x_2\}$ is adjacent to one of $x_1$ and $x_2$, say $x_2$. Then $x_3\not\sim x_1$ and $x_3\sim y_4$. Furthermore, we have $\sigma(x_3x_2)\cdot \sigma(x_3y_4)=+$ since otherwise $i_-(\Gamma)\geq i_-(\Gamma[V(C_5^\sigma\cup \{x_2, x_3\}])=4>\lceil\frac{5}{2}\rceil$. Without loss of generality, set $\sigma(x_3x_2)=\sigma(x_3y_4)=+$. We claim $N_1(v,C_g^\sigma)=\{x_1, x_2, x_3\}$ as follows. Suppose not, let $x_4\in N_1(v,C_g^\sigma)\setminus \{x_1, x_2, x_3\}$. Then $x_4$ is adjacent exactly one of $x_1$ and $x_3$, say $x_3$, since otherwise a contradiction $i_-(\Gamma)\geq i_-(\Gamma[V(C_5^\sigma)\cup \{x_1, x_2, x_3, x_4\}])=4>3=\lceil\frac{5}{2}\rceil$ or $g\leq 4$ arise. Now $x_4\sim y_5$ and $\sigma(x_4x_3)=\sigma(x_4y_5)=+$ by above discussion. Therefore, $\Gamma[V(y_1,\ldots, y_5, x_1,\ldots, x_4]\cong H_1^\sigma$ (see Fig. \ref{fig-4}), which leads to a contradiction $i_-(\Gamma)\geq i_-(H_1^\sigma)=4>\lceil\frac{5}{2}\rceil$ by Lemma \ref{lem-2-1-1}. Hence, $N_1(v,C_g^\sigma)=\{x_1, x_2, x_3\}$, and so $\Gamma\cong\Gamma_3$ (see Fig. \ref{fig-5}).

{\flushleft\bf Case 2.} $g=6$ and $C_6^\sigma$ is balanced.

We may assume that the sign of each edge in $C_6^\sigma$ is positive by Theorem \ref{thm-1-1}. Then $\Gamma[V(C_6^\sigma)\cup\{x_1,x_2\}]\cong (B(4,4,4), \sigma_1)$ (see Fig. \ref{fig-3}). In order of $y_1, y_4, y_3, y_2, y_5, y_6, x_1,$ $x_2$, $A((B(4,4,4), \sigma_1))$ can be written as
$$A((B(4,4,4), \sigma_1))=\left(
        \begin{array}{cccccccc}
          0 & 1 & 0 & 0 & 0 & 1 & a_{17}^\sigma & 0\\
          1 & 0 & 1 & 0 & 0 & 0 & 0 & 0\\
          0 & 1 & 0 & 1 & 0 & 0 & 0 & 0\\
          0 & 0 & 1 & 0 & 1 & 0 & 0 & a_{48}^\sigma\\
          0 & 0 & 0 & 1 & 0 & 1 & 0 & 0\\
          1 & 0 & 0 & 0 & 1 & 0 & 0 & 0\\
          a_{17}^\sigma & 0 & 0 & 0 & 0 & 0 & 0 & a_{78}^\sigma\\
          0 & 0 & 0 & a_{48}^\sigma & 0 & 0 & a_{78}^\sigma & 0\\
        \end{array}
      \right).$$
Applying elementary congruence matrix operations on $A((B(4,4,4), \sigma_1))$, we get that $A$ $((B(4,4,4),\sigma_1))$ is congruent to
$$B=\left(
        \begin{array}{cccccccc}
          0 & 1 & 0 & 0 & 0 & 0 & 0 & 0\\
          1 & 0 & 0 & 0 & 0 & 0 & 0 & 0\\
          0 & 0 & 0 & 1 & 0 & 0 & 0 & 0\\
          0 & 0 & 1 & 0 & 0 & 0 & 0 & 0\\
          0 & 0 & 0 & 0 & 0 & 2 & 0 & 0\\
          0 & 0 & 0 & 0 & 2 & 0 & 0 & 0\\
          0 & 0 & 0 & 0 & 0 & 0 & 0 & a_{78}^\sigma+\frac{a_{17}^\sigma a_{48}^\sigma}{2}\\
          0 & 0 & 0 & 0 & 0 & 0 & a_{78}^\sigma+\frac{a_{17}^\sigma a_{48}^\sigma}{2} & 0\\
        \end{array}
      \right).$$
The matrix $B$ has four negative eigenvalues by a simple observation. Then $i_-(\Gamma)\geq i_-(B(4,$ $4,4),\sigma_1)=4>\lceil\frac{6}{2}\rceil=3$ following from Lemma \ref{lem-2-0}, a contradiction.
\end{proof}

\begin{figure}[t]
    \centering
    \includegraphics[width=0.8\linewidth]{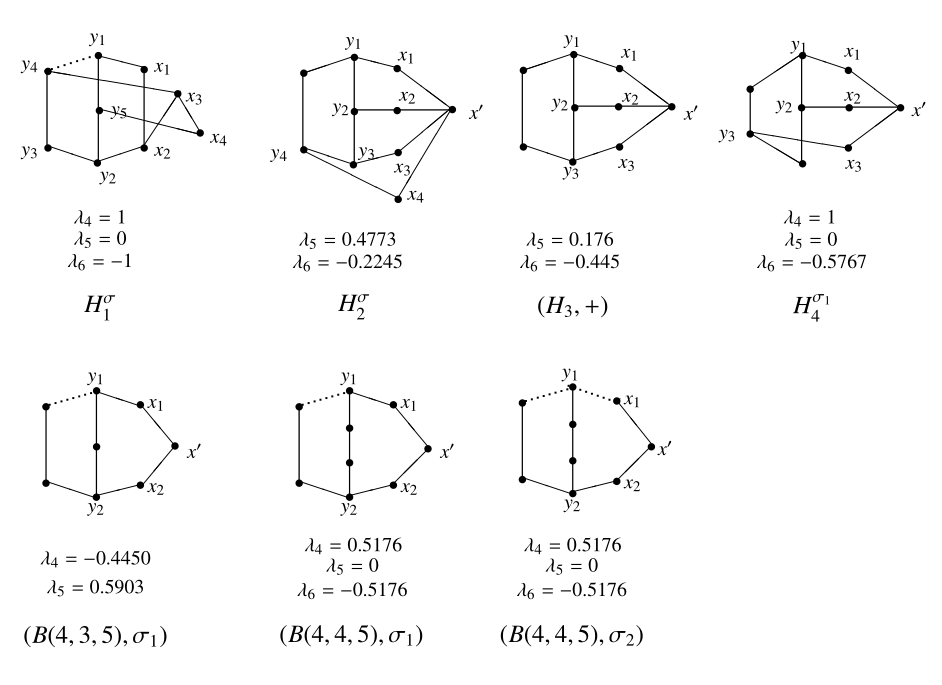}
    \caption{The signed graphs $H_1^\sigma$, $H_2^\sigma$, $(H_3,+)$, $H_4^{\sigma_1}$, $(B(4,3,5),\sigma_1)$, $(B(4,4,5),\sigma_1)$\ and\ $(B(4,4,5),\sigma_2)$.}
    \label{fig-4}
    \end{figure}
\begin{figure}[t]
    \centering
    \includegraphics[width=0.95\linewidth]{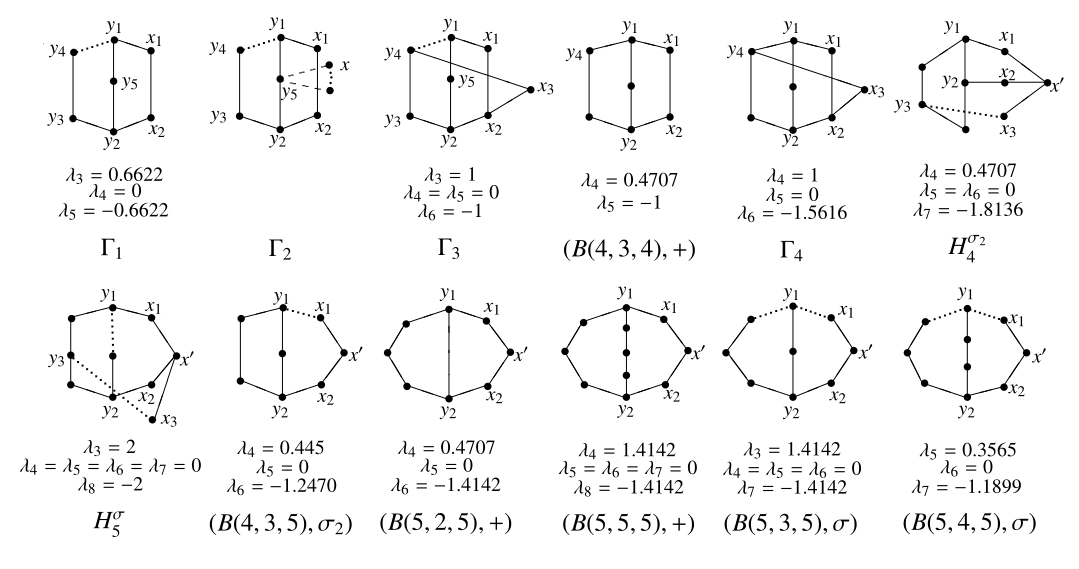}
    \caption{The signed graphs $\Gamma_1$, $\Gamma_2$, $\Gamma_3$, $(B(4,3,4),+)$, $\Gamma_4$, $H_4^{\sigma_2}$, $H_5^\sigma$, $(B(4,3,5),\sigma_2)$, $(B(5,2,5),$ $+)$, $(B(5,5,5),+)$, $(B(5,3,5),\sigma)$\ and\ $(B(5,4,5),\sigma)$.}
    \label{fig-5}
    \end{figure}
\begin{thm}\label{thm-3-4}
Let $\Gamma=(G,\sigma)$ be a connected signed graph with girth $g\geq 4$, and there exists no signed cycle $C_g^\sigma$ of $\Gamma$ satisfying $i_-(C_g^\sigma)=\lceil\frac{g}{2}\rceil$. Suppose that $\Gamma$ is not a canonical signed unicyclic graph. Then $i_-(\Gamma)=\lceil\frac{g}{2}\rceil$ if and only if $\Gamma$ is isomorphic to one of the following signed graphs:
\setlength{\itemsep}{0pt}
\item (1) The signed graphs with girth 4 obtained from $P_4^\sigma$, $P_5^\sigma$, balanced $C_5^\sigma$ or unbalanced $C_6^\sigma$ by adding twin vertices;
\item (2) The signed graphs obtained by joining a vertex of signed cycle $C_g^\sigma$ to the center of a star $K_{1,r}$, where $g\equiv 0,1(\bmod\ 4)$ if $C_g^\sigma$ is balanced and $g\equiv 2,3(\bmod\ 4)$ if $C_g^\sigma$ is unbalanced;
\item (3) $(B(4,3,4),+)$, $\Gamma_4$, $H_4^{\sigma_2}$, $H_5^{\sigma}$, $(B(4,3,5),\sigma_2)$, $(B(5,2,5),+)$, $(B(5,5,5),+)$, $(B(5,3,5),\sigma)$ and $(B(5,4,5),\sigma)$ (see Fig. \ref{fig-5}).
\end{thm}

\begin{proof}
The connected signed graph $\Gamma$ with girth $g\geq 4$ satisfies $i_-(\Gamma)=\lceil\frac{g}{2}\rceil=2$ if and only if $g=4$. Because $\Gamma$ contains no unbalanced 4-cycle due to the hypothesis. By Corollary \ref{cor-1} (2), $\Gamma$ can be obtained from $P_4^\sigma$, $P_5^\sigma$, balanced $C_5^\sigma$ or unbalanced $C_6^\sigma$ by adding twin vertices. It suffice to show the case $g\geq 5$ as follows.

The signed graphs shown in (3) satisfy the equality by a simple operation. Suppose that $\Gamma$ is signed graph shown in (2). By deleting a pendant vertex and center vertex of $K_{1,r}$ from $\Gamma$, we have
$$i_-(\Gamma)=i_-(C_g^\sigma)+1=(\lceil\frac{g}{2}\rceil-1)+1=\lceil\frac{g}{2}\rceil$$
due to Lemmas \ref{lem-2-1} and \ref{lem-2-2}. Conversely, let $\Gamma=(G,\sigma)$ be a connected signed graph with girth $g\geq 5$ and we consider the following two cases:
{\flushleft\bf Case 1.} $N_2(v,C_g^\sigma)=\emptyset$.

Since $g\geq 5$, any vertex of $N_1(v,C_g^\sigma)$ is adjacent to exactly one vertex of $V(C_g^\sigma)$ by Lemma \ref{lem-2-5}. Then
$\Gamma[N_1(v,C_g^\sigma)]$ has at least one edge, say $x_1x_2$, because $\Gamma$ is not canonical signed unicyclic graph. Set $x_1\sim y_1\in V(C_g^\sigma)$ and $x_2\sim y_2\in V(C_g^\sigma)$. Now there exists a path of length 3 from $y_1$ to $y_2$ vertex-disjoint to $C_g^\sigma$. Appealing again to Lemma \ref{lem-2-5}, we have $\lceil\frac{g}{2}\rceil\leq 3$. The following two subcases can be consider:

{\flushleft\bf Subcase 1.}  $g=5$ and there exists a balanced $C_5^\sigma$ in $\Gamma$.

Clearly, $\Gamma[V(C_5^\sigma)\cup \{x_1, x_2\}]\cong (B(4, 3, 4),+)$ (see Fig. \ref{fig-5}) since $\Gamma$ contains no unbalanced 5-cycle. Assume that $N_1(v, C_5^\sigma)\setminus\{x_1, x_2\}=\emptyset$. Then, $\Gamma\cong (B(4, 3, 4), +)$, as desired. Assume that $x_3\in N_1(v,C_5^\sigma)\setminus\{x_1, x_2\}\neq\emptyset$. If $x_3$ is not adjacent to $x_1$ and $x_2$ at the same time, then, by a simple observation, a contradiction $i_-(\Gamma)\geq i_-(\Gamma[V(C_5^\sigma)\cup \{x_1,x_2,x_3\}])=4>3=\lceil\frac{5}{2}\rceil$ arise no matter what vertex of $V(C_5^\sigma)$ is adjacent to $x_3$. Hence, $x_3$ is adjacent to one of $x_1$, $x_2$, say $x_2$, and $x_3\sim y_4$ since $g=5$. We claim that
$N_1(v,C_5^\sigma)=\{x_1, x_2, x_3\}$ as follows. To see this, suppose that $x_4\in N_1(v,C_5^\sigma)$ is different from $x_1, x_2, x_3$. As noted above, $x_4$ is adjacent to exactly one of $x_1$, $x_2$ and $x_2$, $x_3$, respectively. This gives $x_4\sim x_2$ since otherwise $g\leq 4$. However, it also leads to a contradiction $g\leq 4$ because $x_1$, $x_2$, $x_3$, $x_4$ are all adjacent to one vertex of $V(C_5^\sigma)$. Hence, $N_1(v,C_5^\sigma)=\{x_1, x_2, x_3\}$, and so $\Gamma\cong\Gamma_4$ (see Fig. \ref{fig-5}) because $\Gamma$ contains no unbalanced 5-cycle.

{\flushleft\bf Subcase 2.} $g=6$ and there exists an unbalanced $C_6^\sigma$ in $\Gamma$.

We may assume that exactly one edge of $C_6^\sigma$, say $y_1y_4$, is negative by Lemma \ref{lem-2-11}. Then $\Gamma[V(C_6^\sigma)\cup\{x_1,x_2\}]\cong (B(4,4,4), \sigma_2)$ (see Fig. \ref{fig-3}). Furthermore, since $\Gamma$ contains no balanced signed 6-cycle, there is exactly one negative edge in $x_1y_1, x_1x_2$ and $x_2y_2$ by Lemma \ref{lem-2-11}. But now $\{x_1, x_2,y_1, y_2, y_3, y_4\}$ induces a balance 6-cycle, a contradiction.

{\flushleft\bf Case 2.} $N_2(v,C_g^\sigma)\neq\emptyset$.

We firstly give three claims in this case.

{\flushleft\bf Claim 1.} $N_2(v,C_g^\sigma)$ is an independent set, and each vertex of $N_1(v,C_g^\sigma)$ is adjacent to at least one vertex of $N_2(v,C_g^\sigma)$.
\begin{proof}
Suppose, by way of contradiction, that $N_2(v,C_g^\sigma)$ is not an independent set. Let $x_1',x_2'\in N_2(v,C_g^\sigma)$ and $x_1'\sim x_2'$. We may assume $x_1'\sim x_1\in N_1(v,C_g^\sigma)$. Since $g\neq 3$, $x_1 \not\sim x_2'$. Furthermore, by Lemma \ref {lem-2-5}, $x_1$ is adjacent to exactly one vertex of $V(C_g^\sigma)$, say $y$, because $g\geq 5$. A contradiction
$i_-(\Gamma)\geq i_-(\Gamma[V(C_g^\sigma)\cup \{x_1',x_2',x_1\}])=i_-(\Gamma[V(C_g^\sigma) \setminus \{y\}])+2=\lfloor\frac{g-1}{2}\rfloor+2>\left \lceil \frac{g}{2}\right\rceil$ follows from Lemmas \ref{lem-2-1-1}, \ref{lem-2-1} and \ref{lem-2-2}. Hence, $N_2(v,C_g^\sigma)$ is an independent set.

On the other hand, let $x\in N_1(v,C_g^\sigma)$, $x\sim y\in V(C_g^\sigma)$ and $x$ is not adjacent to all vertices of $N_2(v,C_g^\sigma)$. Set $x'\in N_2(v,C_g^\sigma)$ and $x'\sim x_1\in N_1(v,C_g^\sigma)$. According to Lemmas \ref{lem-2-1-1}, \ref{lem-2-1} and \ref{lem-2-2}, we get a contradiction
$$i_-(\Gamma)\geq i_-(G[V(C_g^\sigma)\cup \{x_1, x, x'\}])=i_-(G[V(C_g^\sigma)\setminus\{y\}])+2=\lfloor\frac{g-1}{2}\rfloor+2>\lceil\frac{g}{2}\rceil.$$
Hence, each vertex of $N_1(v,C_g^\sigma)$ is adjacent to at least one vertex of $N_2(v,C_g^\sigma)$.
\end{proof}

{\flushleft\bf Claim 2.}  The vertex of $N_2(v,C_g^\sigma)$ is adjacent to at most three vertices of $N_1(v,C_g^\sigma)$. Furthermore, if one vertex of $N_2(v,C_g^\sigma)$ is adjacent to two vertices of $N_1(v,C_g^\sigma)$, then $|N_2(v,C_g^\sigma)|=1$.
%this means $g\geq 8$ and $g\geq6$
\begin{proof}
On the contrary, suppose that $x'\in N_2(v,C_g^\sigma)$ is adjacent to four vertices $x_1,x_2,$ $x_3,x_4$ of $N_1(v,C_g^\sigma)$. Then, $x_1,x_2,x_3,x_4$ have no
common neighbors in $V(C_g^\sigma)$ since $g\geq 5$. Set $x_i\sim y_i\in V(C_g^\sigma)$ for $i\in \{1,\ldots, 4\}$. Let $l$ be length of the shortest path in $C_g^\sigma$ between $y_1,y_2,y_3,y_4$. Then, $l\leq\lfloor\frac{g}{4}\rfloor$. Now we can found a cycle of length at most $\lfloor\frac{g}{4}\rfloor+4$ in $\Gamma$. Clearly, $\lfloor\frac{g}{4}\rfloor+4\geq g$, in other word, $\lceil\frac{3g}{4}\rceil \leq 4$, which implies that $C_g^\sigma$ is a balanced 5-cycle. Hence, $\Gamma[V(C_g^\sigma)\cup\{x',x_1,x_2,x_3,x_4\}]\cong H_2^\sigma$ (see Fig. \ref{fig-4}) because $\Gamma$ contains no unbalanced 5-cycle. This leads to a contradiction $i_-(\Gamma)\geq i_-(H_2^\sigma)=4>3=\lceil\frac{5}{2}\rceil$ by Lemma \ref{lem-2-1-1} and we are done.

Next, suppose that $x_1'\in N_2(v,C_g^\sigma)$ is adjacent to $x_1,x_2\in N_1(v,C_g^\sigma)$. If there exists $x_2'\in N_2(v,C_g^\sigma)$ different from $x_1'$ and $x_2'\sim x_3\in
N_1(v,C_g^\sigma)$. Without loss of generality, we may assume that $x_3\not= x_1$. Since $x_1'\not\sim x_2'$ by Claim 1, we have $$i_-(\Gamma[V(C_g^\sigma)\cup\{x_1,x_3,x_1',x_2'\}])=i_-(\Gamma[V(C_g^\sigma)\cup\{x_1,x_1'\}])+1=i_-(C_g^\sigma)+2.$$
Then Lemma \ref{lem-2-1-1} implies a contradiction $i_-(\Gamma)\geq i_-(C_g^\sigma)+2=(\lceil\frac{g}{2}\rceil-1)+2>\lceil\frac{g}{2}\rceil$. Hence, $|N_2(v,C_g^\sigma)|=1$.
\end{proof}

{\flushleft\bf Claim 3.} $N_i(v,C_g^\sigma)=\emptyset$ for $i\geq 3$.

\begin{proof}
It suffice to prove $N_3(v,C_g^\sigma)=\emptyset$. On the contrary, suppose that $x''\in N_3(v,C_g^\sigma)\neq\emptyset$. Then, we may assume $x''\sim x'\in N_2(v,C_g^\sigma)$ and $x'\sim x\in N_1(v,C_g^\sigma)$. Recall that $x$ is adjacent to exactly one vertex of $V(C_g^\sigma)$, say $y$, by Lemma \ref{lem-2-5}. According to Lemmas \ref{lem-2-1-1}, \ref{lem-2-1} and \ref{lem-2-2}, we get a  contradiction
$$i_-(\Gamma)\geq i_-(\Gamma[V(C_g^\sigma)\cup\{x,x',x''\}])=i_-(\Gamma[V(C_g^\sigma)\setminus\{y\}])+2=\lfloor\frac{g-1}{2}\rfloor+2>\lceil\frac{g}{2}\rceil.$$
\end{proof}

Now we consider the following three cases:
{\flushleft\bf Subcase 1.} One vertex of $N_2(v,C_g^\sigma)$ is adjacent to three vertices $x_1, x_2, x_3$ of $N_1(v,C_g^\sigma)$.

We may assume $N_2(v,C_g^\sigma)=\{x'\}$ by Claim 2, then $N_1(v,C_g^\sigma)=\{x_1, x_2, x_3\}$ by Claim 1. Recall that any vertex of $N_1(v,C_g^\sigma)$ is adjacent to one vertex of $V(C_g^\sigma)$ by Lemma \ref{lem-2-5}. Then, $x_1,x_2,x_3$ have no common neighbors in $V(C_g^\sigma)$ because $g\geq 5$. Set $y_1,y_2,y_3\in V(C_g^\sigma)$ and $x_i\sim y_i$ for $i\in \{1,2,3\}$. By the manner similar to that of Claim 2, we have $\lceil\frac{2g}{3}\rceil \leq 4$. Then $C_g^\sigma$ is a balanced 5-cycle or an unbalanced 6-cycle. Since $N_r(v, C_g^\sigma)=\emptyset$ for $r\geq 3$ by Claim 3, $\Gamma$ is switching equivalent to $(H_3, +)$ (see Fig. \ref{fig-4}), $H_4^\sigma$ (see Fig. \ref{fig-3}) or $H_5^\sigma$ (see Fig. \ref{fig-4}). Note that $H_4^\sigma$ is switching equivalent to $H_4^{\sigma_1}$ (see Fig. \ref{fig-4}) or $H_4^{\sigma_2}$ (see Fig. \ref{fig-5}). By a simple operation, we have $i_-(H_3, +)=i_-(H_4^{\sigma_1})=4\geq 3=\lceil\frac{5}{2}\rceil$ and $i_-(H_4^{\sigma_2})=i_-(H_5^\sigma)=3$. Hence $\Gamma\cong H_4^{\sigma_2}$ or $H_5^{\sigma}$.

{\flushleft\bf Subcase 2.} One vertex of $N_2(v,C_g^\sigma)$ is adjacent to exactly two vertices $x_1, x_2$ of $N_1(v,C_g^\sigma)$.

We also set $N_1(v,C_g^\sigma)=\{x_1, x_2\}$ by Claim 1 and $N_2(v,C_g^\sigma)=\{x'\}$  by Claim 2. Let $x_1\sim y_1\in V(C_g^\sigma)$, $x_2\sim y_2\in V(C_g^\sigma)$ and $y_1\neq y_2$ since $g\geq 5$. Consequently, there exists a path of length 4 from $y_1$ to $y_2$ vertex-disjoint to $V(C_g^\sigma)$, it gives $\lceil\frac{g}{2}\rceil\leq4$ by Lemma \ref{lem-2-5}. Thus, $g\leq 8$, and so $C_g^\sigma$ is a balanced 5-cycle, 8-cycle, or a unbalanced 6-cycle, 7-cycle. Now $\Gamma$ is switching equivalent to one of $(B(4,3,5),\sigma)$ (see Fig. \ref{fig-3}), $(B(5,2,5), +)$ (see Fig. \ref{fig-5}), $(B(5,5,5),+)$ (see Fig. \ref{fig-5}), $(B(5,3,5),\sigma)$ (see Fig. \ref{fig-5}), $(B(4,4,5),\sigma)$ (see Fig. \ref{fig-3}) and $(B(5,4,5),\sigma)$ (see Fig. \ref{fig-5}). Note that $(B(4,3,5),\sigma)$ is switching equivalent to $(B(4,3,5),\sigma_1)$ (see Fig. \ref{fig-4}) or $(B(4,3,5),\sigma_2)$ (see Fig. \ref{fig-5}), $(B(4,4,5),\sigma)$ is switching equivalent to $(B(4,4,5),\sigma_1)$ or $(B(4,4,5),\sigma_2)$ (see Fig. \ref{fig-4}).
By a simple operation, we get $i_-((B(4,3,5),\sigma_1))=4>\lceil\frac{5}{2}\rceil$, $i_-((B(4,3,5),\sigma_2))=3=\lceil\frac{5}{2}\rceil$, $i_-((B(5,2,5),+))=3=\lceil\frac{5}{2}\rceil$, $i_-((B(5,5,5),+))=4=\lceil\frac{8}{2}\rceil$, $i_-((B(5,3,5),\sigma))=3=\lceil\frac{6}{2}\rceil$, $i_-((B(4,4,5),\sigma_1))=i_-((B(4,4,5),\sigma_2))=4>\lceil\frac{6}{2}\rceil$, $i_-((B(5,4,5),\sigma))=4=\lceil\frac{7}{2}\rceil$. Hence $\Gamma\cong (B(4,3,5),\sigma_2)$ $, (B(5,2,5),+), (B(5,5,5),+)$, $(B(5,3,5),\sigma)$ or $(B(5,4,5),\sigma)$.

{\flushleft\bf Subcase 3.} Any vertex of $N_2(v,C_g^\sigma)$ is adjacent to exactly one vertex of $N_1(v,C_g^\sigma)$.

Recall that $N_2(v,C_g^\sigma)$ is an independent set by Claim 1. Firstly, we show that all vertices of $N_2(v,C_g^\sigma)$ have same neighbors in $N_1(v,C_g^\sigma)$. In fact, suppose that $x_1',x_2'\in N_2(v,C_g^\sigma)$ and $x_1'\sim x_1$, $x_2'\sim x_2$, where $x_1,x_2\in N_1(v,C_g^\sigma)$ and $x_1\neq x_2$. Then, we have $$i_-(\Gamma)\geq i_-(\Gamma[V(C_g^\sigma)\cup\{x_1,x_2,x_1',x_2'\}])=i_-(C_g^\sigma)+2=(\lceil\frac{g}{2}\rceil-1)+2>\lceil\frac{g}{2}\rceil$$ by Lemmas \ref{lem-2-1-1} and \ref{lem-2-2}, a contradiction. Hence, all vertices of $N_2(v,C_g^\sigma)$ are adjacent to one vertex of $N_1(v,C_g^\sigma)$, say $x_1$. Furthermore, we have $N_1(v,C_g^\sigma)=\{x_1\}$ by Claim 1. Hence $\Gamma$ is obtained by joining a vertex of cycle $C_g^\sigma$ to the center of a star $K_{1,r}$.
\end{proof}

\section{Connected signed graphs with given positive inertia index and given girth}
Let $\Gamma=(G, \sigma)$ be a connected signed graph. By replacing the signature $\sigma$ of $\Gamma$ with $-\sigma$, one can obtain a signed graph $-\Gamma=(G, -\sigma)$
known as the \emph{negation} of $\Gamma$. By elementary knowledge of linear algebra, the negative inertia index of $\Gamma$ is equal to the positive inertia index of $-\Gamma$. Hence the signed graphs with given positive inertia index and given girth are determined in this section.

\begin{thm}\label{thm-4-1}
Let $\Gamma=(G,\sigma)$ be a connected signed graph with girth $g$ and $C_g^\sigma$ be a shortest cycle in $\Gamma$. Then $i_+(\Gamma)\geq\lceil\frac{g}{2}\rceil-1$. Equality holds if and only if $\Gamma$ satisfies one of the following conditions:
\begin{enumerate}[(1)]
\setlength{\itemsep}{0pt}
\item $\Gamma\cong C_g^\sigma$, where $C_g^\sigma$ is balanced and $g\equiv 0,3(\bmod\ 4)$, or unbalanced and $g\equiv 1,2(\bmod\ 4)$;
\item $\Gamma\cong (K^\sigma_{n_1,n_2,\ldots,n_l}, +)$.
\end{enumerate}
\end{thm}

\begin{thm}\label{thm-4-2}
Let $\Gamma$ be a canonical signed unicyclic graph with girth $g$ and the unique cycle $C_g^\sigma$. Then, the following statements hold:
\begin{enumerate}[(1)]
\setlength{\itemsep}{0pt}
\item If $\Gamma$ is a cycle, then $i_+(\Gamma)=\lceil \frac{g}{2}\rceil$ if and only if $\Gamma\cong C_g^\sigma$, where $C_g^\sigma$ is balanced and $g\equiv 1,2(\bmod\ 4)$, or unbalanced and $g\equiv 0,3(\bmod\ 4)$;
\item If $\Gamma$ is not a cycle and $g\equiv 1,3(\bmod\ 4)$, then $i_+(\Gamma)=\lceil \frac{g}{2}\rceil$ if and only if $\Gamma$ has one or more attached pendant stars such that exactly one path between any two major vertices of $V(C_g^\sigma)$ has even order;
\item If $\Gamma$ is not a cycle and $g\equiv 0,2(\bmod\ 4)$, then $i_+(\Gamma)=\lceil \frac{g}{2}\rceil$ if and only if $\Gamma$ has one or more attached pendant stars such that all paths between any two major vertices of $V(C_g^\sigma)$ have odd order.
\end{enumerate}
\end{thm}

\begin{thm}\label{thm-4-3}
Let $\Gamma=(G,\sigma)$ be a connected signed graph with girth $g\geq 4$. Suppose that $\Gamma$ is not canonical signed unicyclic graph. Then $i_+(\Gamma)=\lceil\frac{g}{2}\rceil$ if and only if $\Gamma$ is isomorphic to one of the following signed graphs:
\begin{enumerate}[(1)]
\setlength{\itemsep}{0pt}
\item The signed graphs with girth 4 obtained from $P_4^\sigma$, $P_5^\sigma$, unbalanced $C_5^\sigma$, unbalanced $C_6^\sigma$, unbalanced $C_4^\sigma$, $G_1^\sigma$ or $G_2^\sigma$ by adding twin vertices;
\item The signed graphs obtained by joining a vertex of cycle $C_g^\sigma$ to the center of a star $K_{1,r}$, where $g\equiv 0,3(\bmod\ 4)$ if $C_g^\sigma$ is balanced and $g\equiv 1,2(\bmod\ 4)$ if $C_g^\sigma$ is unbalanced;
\item The signed graphs $\Gamma_1, \Gamma_2$, $\Gamma_3$, $(B(4,3,4),-)$, $-\Gamma_4$, $-H_4^{\sigma_2}$, $H_5^{\sigma}$, $-(B(4,3,5),\sigma)$, $(B(5,$ $2,5),-)$, $(B(5,3,5),\sigma)$, $-(B(5,4,5),\sigma)$ and $(B(5,5,5),+)$ (see Fig. \ref{fig-5}).
\end{enumerate}
\end{thm}

\section{Connected signed graphs with given nullity and given girth}
Evidently, the relation $i_+(\Gamma)+i_-(\Gamma)+\eta(\Gamma)=n$ holds for a connected signed graph $\Gamma$ of order $n$. By using the conclusions of inertia indices
in above two sections, we derive connected signed graphs with given nullity and given girth in this section, which relate the conclusion of \cite{Q.Wu}.

Suppose that the inertia indices of signed graph $\Gamma$ are all $\lceil \frac{g}{2}\rceil-1$. Then $\eta(\Gamma)=n-g+1$ if and only if $g$ is an odd, and $\eta(\Gamma)=n-g+2$ if and only if $g$ is an even. Hence, $\eta(\Gamma)=n-g+2$ holds if and only if the above conditions holds. By using the lower bounds of inertia indices in Theorems \ref{thm-3-1} and \ref{thm-4-1}, and taking out the identical extremal signed graphs corresponding to the lower bounds of inertia indices for even $g$, Theorem \ref{thm-5-1} follows.

\begin{thm}\label{thm-5-1}
Let $\Gamma=(G,\sigma)$ be a connected signed graph with girth $g$ and $C_g^\sigma$ be a shortest cycle in $\Gamma$. Then $\eta(\Gamma)\leq n-g+2$. Equality holds if and only if $\Gamma$ is a balanced $C_g^\sigma$ for $g\equiv 0(\bmod\ 4)$, an unbalanced $C_g^\sigma$ for $g\equiv 2(\bmod\ 4)$ or $\Gamma\cong (K_{n_1,n_2}, +)$.
\end{thm}

Suppose that the inertia indices of a signed graph $\Gamma$ are $\lceil \frac{g}{2}\rceil-1$ and $\lceil \frac{g}{2}\rceil$, respectively. Then $\eta(\Gamma)=n-g+1$ if and only if $g$ is an even, and $\eta(\Gamma)=n-g$ if and only if $g$ is an odd. Suppose that the inertia indices of a signed graph $\Gamma$ are all $\lceil \frac{g}{2}\rceil$. Then  $\eta(\Gamma)=n-g-1$ if and only if $g$ is an odd, and $\eta(\Gamma)=n-g$ if and only if $g$ is an even. Hence, $\eta(\Gamma)=n-g+1$ holds if and only if $i_+(\Gamma)=\lceil \frac{g}{2}\rceil-1$ and $i_-(\Gamma)=\lceil \frac{g}{2}\rceil-1$ for odd $g$, or $i_+(\Gamma)=\lceil \frac{g}{2}\rceil$ and $i_-(\Gamma)=(\lceil \frac{g}{2}\rceil-1)$ for even $g$, or $i_-(\Gamma)=\lceil \frac{g}{2}\rceil$ and $i_+(\Gamma)=\lceil \frac{g}{2}\rceil-1$ for even $g$. Unfortunately, there exists no signed graphs satisfying these conditions.

Similarly, $\eta(\Gamma)=n-g$ holds if and only if $i_+(\Gamma)=(\lceil \frac{g}{2}\rceil)$ and $i_-(\Gamma)=\lceil \frac{g}{2}\rceil-1$ for odd $g$, or $i_-(\Gamma)=\lceil \frac{g}{2}\rceil$ and $i_+(\Gamma)=\lceil \frac{g}{2}\rceil-1$ for odd $g$, or $i_+(\Gamma)=\lceil \frac{g}{2}\rceil$ and $i_-(\Gamma)=\lceil \frac{g}{2}\rceil$ for even $g$. Hence connected signed graph $\Gamma$ with girth $g$ and $\eta(\Gamma)=n-g$ is characterized in following two theorems.

\begin{thm}\label{thm-5-2}
Let $\Gamma$ be a canonical signed unicyclic graph with girth $g$ and the unique cycle $C_g^\sigma$. Then, the following statements hold:
\begin{enumerate}[(1)]
\setlength{\itemsep}{0pt}
\item If $\Gamma$ is a cycle, then $\eta(\Gamma)=n-g$ if and only if $\Gamma\cong C_g^\sigma$, where $g$ is an odd, or $C_g^\sigma$ is balanced and $g\equiv 2(\bmod\ 4)$, or unbalanced and $g\equiv 0(\bmod\ 4)$;
\item If $\Gamma$ is not a cycle, then $\eta(\Gamma)=n-g$ if and only if $g$ is an even and $\Gamma$ has one or more attached pendant stars such that all paths between any two major vertices of $V(C_g^\sigma)$ have odd order.
\end{enumerate}
\end{thm}

\begin{thm}\label{thm-5-3}
Let $\Gamma=(G,\sigma)$ be a connected signed graph with girth $g\geq 4$. Suppose that $\Gamma$ is not canonical signed unicyclic graph. Then $\eta(\Gamma)=n-g$ if and only if $\Gamma$ is isomorphic to one of the following signed graphs:
\begin{enumerate}[(1)]
\setlength{\itemsep}{0pt}
\item The signed graphs with girth 4 obtained from $P_4^\sigma$, $P_5^\sigma$, unbalanced $C_6^\sigma$, unbalanced $C_4^\sigma$, $G_1^\sigma$ or $G_2^\sigma$ by adding twin vertices;
\item The signed graphs obtained by joining a vertex of cycle $C_g^\sigma$ to the center of a star $K_{1,r}$, where $g\equiv 0(\bmod\ 4)$ if $C_g^\sigma$ is balanced and $g\equiv 2(\bmod\ 4)$ if $C_g^\sigma$ is unbalanced;
\item The signed graphs $H_5^{\sigma}$, $(B(5,3,5),\sigma)$ and $(B(5,5,5),+)$ (see Fig. \ref{fig-5}).
\end{enumerate}
\end{thm}

Since the condition $g\geq 4$ in Theorem \ref{thm-5-3}, there are two types connected signed graphs that appear in \cite{Q.Wu} but do not appear in Theorem \ref{thm-5-3}. They are signed graphs with girth 3 obtained from balanced $C_3^\sigma$ or unbalanced $C_3^\sigma$ by adding twin vertices.
\vspace{3mm}

\noindent \textbf{Declaration of competing interest}\vspace{3mm}

The authors declare that there are no conflict of interests.\vspace{3mm}

\noindent \textbf{Data availability}\vspace{3mm}

No data was used for the research described in the article.

%\bibitem{F.L.Tian} F.L. Tian, D.Y. Wang, M. Zhu, A characterizerization of signed planar graphs with rank at most 4, Linear and Multilinear Algebra. 64(5)(2016)807--817.

\end{document}